\documentclass[3p]{elsarticle}

\usepackage[utf8]{inputenc}
\usepackage{amsmath}
\usepackage{amssymb}
\usepackage{amsthm}
\usepackage{stmaryrd}
\usepackage{pgfplots}
\usepackage{pgfplotstable}
\usepackage{tikz}
\usepackage{algorithm}
\usepackage{xcolor}
\usepackage{algorithmicx}
\usepackage{algpseudocode}
\usepackage{placeins}

\newtheorem{thm}{Theorem}
\newtheorem{lem}{Lemma}
\newtheorem{ass}{Assumption}
\theoremstyle{definition}
\newtheorem{rem}{Remark}

\pgfplotsset{compat=1.17}

\newcommand{\Div}[0]{\mathbf{div}\,}
\newcommand{\bs}[1]{\boldsymbol{#1}}
\newcommand{\Sig}[0]{\bs{\sigma}}

\newcommand{\Eps}[0]{\bs{\varepsilon}}
\newcommand{\U}[0]{\bs{u}}
\newcommand{\V}[0]{\bs{v}}
\newcommand{\W}[0]{\bs{w}}
\newcommand{\Lam}[0]{\bs{\lambda}}
\newcommand{\Mu}[0]{\bs{\mu}}
\newcommand{\X}[0]{\bs{\xi}}
\newcommand{\jump}[1]{\left\llbracket #1 \right\rrbracket}
\newcommand{\Bf}[0]{\mathcal{B}}
\newcommand{\Lf}[0]{\mathcal{L}}

\newcommand{\N}[0]{\bs{n}}

\newcommand{\F}[0]{\bs{f}}
\newcommand{\Z}[0]{\bs{0}}
\newcommand{\Th}[0]{\mathcal{T}_h}

\newcommand{\Gh}[0]{\mathcal{G}_h}

\newcommand{\enorm}[1]{{\left\vert\kern-0.25ex\left\vert\kern-0.25ex\left\vert #1 
            \right\vert\kern-0.25ex\right\vert\kern-0.25ex\right\vert}}

\begin{document}

\title{Mortaring for linear elasticity using mixed and stabilized finite elements}

\begin{abstract}


The purpose of this work is to study mortar methods
for linear elasticity using standard
low order finite element spaces.
Based on residual stabilization,
we introduce a stabilized mortar method
for linear elasticity and compare it to
the unstabilized mixed mortar method.
For simplicity,
both methods use a Lagrange multiplier
defined on a trace mesh inherited
from one side of the interface only.
We derive a quasi-optimality estimate for the stabilized method and present the stability criteria of the mixed
$P_1-P_1$ approximation.
Our numerical results demonstrate
the stability and the convergence of the methods
for tie contact problems.
Moreover, the results
show that the mixed method
can be successfully extended to 
three dimensional problems.

\end{abstract}

\author[1]{Tom~Gustafsson}
\author[2]{Peter~Råback}
\author[3]{Juha~Videman}




\maketitle

\section{Introduction}

Mortaring via Lagrange multipliers
is also a prototype of elastic contact problems
besides its use in domain decomposition.
Consider, for example, a linear elastic body $\Omega \subset \mathbb{R}^d$, $d=2,3$, with the displacement field $\U : \Omega \rightarrow \mathbb{R}^d$ and body loading $\F : \Omega \rightarrow \mathbb{R}^d$.  In the absence of other external forces, the principle of virtual work reads as follows:
\[
\int_{\Omega} \Sig(\U) : \Eps(\V)\,\mathrm{d}x = \int_{\Omega} \F \cdot \V \,\mathrm{d}x \quad \forall \V \in V,
\]
where $\Eps(\W) = \frac12(\nabla \W + \nabla \W^T)$ is the infinitesimal strain tensor, $\Sig(\W) = \bs{\mathcal{C}} : \Eps(\W)$ is the linear elastic
stress tensor corresponding to the fourth-order constitutive tensor $\bs{\mathcal{C}}$,
and $V$ is the space of kinematically admissible displacement fields.
Suppose now that $\Omega$
is split into two parts $\Omega_1$ and $\Omega_2$ with the displacements $\U_1 \in V_1$ and $\U_2 \in V_2$,
respectively.  The continuity across the interface $\Gamma = \partial \Omega_1 \cap \partial \Omega_2$ can be enforced using
a vectorial Lagrange multiplier field $\Lam : \Gamma \rightarrow \mathbb{R}^d$, $\Lam \in \varLambda$, through the saddle point formulation
\begin{equation}
    \label{eq:saddleform}
    \left\{
    \begin{alignedat}{5}
      &\int_{\Omega_1} \Sig(\U_1) : \Eps(\V_1)\,\mathrm{d}x && ~ && + \int_{\Gamma} \Lam \cdot \V_1 \,\mathrm{d}s &&= \int_{\Omega_1} \F \cdot \V_1\,\mathrm{d}x \qquad &&\forall \V_1 \in V_1, \\
      &~&&\int_{\Omega_2} \Sig(\U_2) : \Eps(\V_2)\,\mathrm{d}x && - \int_{\Gamma} \Lam \cdot \V_2 \,\mathrm{d}s &&= \int_{\Omega_2} \F \cdot \V_2\,\mathrm{d}x \qquad &&\forall \V_2 \in V_2, \\
      &\int_{\Gamma} \Mu \cdot \U_1 \,\mathrm{d}s &&-\int_{\Gamma} \Mu \cdot \U_2 \,\mathrm{d}s && ~ && = 0 \qquad && \forall \Mu \in \varLambda,
    \end{alignedat}
    \right.
\end{equation}
where $\varLambda$ is the trace space of $V_1$ on the interface $\Gamma$.
In contact mechanics, the linear saddle point formulation \eqref{eq:saddleform} corresponds to a single contact
iteration with zero initial gap
between the bodies $\Omega_1$ and $\Omega_2$.
In fact, the stability results
carry over to the variational inequality
formulations.
Formulation \eqref{eq:saddleform} is sometimes referred to as \emph{tie contact problem}.

In principle, the Lagrange multiplier formulation
can accommodate nonmatching computational meshes and even different finite element spaces
over $\Gamma$,
and, thus, can be made quite generic
with respect to the discretization
of parts $\Omega_1$ and $\Omega_2$.
However, it is well known~\cite{babuvska1973finite} that saddle point system \eqref{eq:saddleform} is
stable only if a certain compatibility criterion, the Babu\v{s}ka--Brezzi condition,
holds for the finite element
discretizations of $(\U_1, \U_2)$ and $\Lam$.
The lack of stability can have various consequences from 
reduction in the asymptotic convergence rate to
a singular linear system.
Moreover, the instability
may manifest itself only for specific mesh configurations.
In dealing with mortar finite element methods,
this has led to the development of special finite element techniques
aiming at
proving stability independently of the finite element mesh.
These techniques include enriching the displacement field with bubble degrees-of-freedom~\cite{hauret2006bv, hauret2007discontinuous},
special biorthogonal finite element bases for the Lagrange multiplier~\cite{wohlmuth2000mortar, wohlmuth2011variationally, fritz2004comparison},
local modification of basis functions at the interface~$\Gamma$~\cite{bernardi1993domain, belgacem1998mortar, belgacem1999mortar},
or modification of the bilinear form using additional stabilizing terms~\cite{liu2010stabilized, gustafsson2019error, gustafsson2022stabilized, heintz2006stabilized}.

Based on the abundance of literature on such special stabilization techniques,
it could be easily inferred that they are
absolutely necessary
for successful approximation of tie contact problems.
However, many existing finite element solvers implement
piecewise-linear finite elements and,
therefore, it is natural
to consider also standard finite elements for the
mixed variational formulation and analyze their stability criteria.
The pair $P_1 - P_0$, i.e.~piecewise-linear displacement and
piecewise-constant Lagrange multiplier,
is not stable, cf.~\cite{wohlmuth2011variationally}.
However, continuous piecewise-linear elements 
for both the displacement and the Lagrange multiplier
can be made stable by special
modifications of the basis functions at Dirichlet boundaries
or
at interface junctions between three or more subdomains, cf.~\cite{belgacem1999mortar, belgacem1998mortar, hauret2006bv}.
We argue that in contact mechanics one is less likely to encounter
such configurations and proceed with the study of
the unmodified $P_1-P_1$ pair,
assuming that the bodies are not fixed
right next to the contact interface.

The purpose of this work is to study
the discretization of the saddle point problem \eqref{eq:saddleform}
using mixed and stabilized finite element methods and investigate
the stability and the convergence of the methods.
In the past, residual-based Barbosa--Hughes
stabilization~\cite{barbosa1991finite} 
has been applied to interface problems, e.g., for the
Poisson equation~\cite{hansbo2005lagrange} as well as
frictionless~\cite{hansbo2016least} and frictional contact~\cite{gustafsson2022stabilized}.
Here we introduce, to our knowledge for the first time, residual stabilized mortar method
for tie contact
or, equivalently, domain decomposition in linear
elasticity.
The method is proven stable
for any $L^2$-conforming finite element spaces which is
shown to lead to a quasi-optimality
error estimate.
Next we consider a mixed method using
a continuous Lagrange multiplier and
present criteria for the
stability of the lowest order $P_1-P_1$ method
under which we are able to prove the Babu\v{s}ka--Brezzi condition.
For additional verification, we present
novel numerical results comparing
the behavior of the mixed
and stabilized methods under uniform refinements,
and unique three-dimensional
results using an implementation of the mixed $P_1-P_1$ method in
Elmer~\cite{malinen2013elmer}, a suite of open source finite element solvers
for multiphysical problems.


\section{Problem formulation}

The boundary $\partial \Omega_i$
is split into the Dirichlet part $\Gamma_{D,i}$, the Neumann part $\Gamma_{N,i}$, $i=1,2$,
and the interface $\Gamma$ between the two domains $\Omega_1$ and $\Omega_2$.
We assume that $\Gamma_{D,i}$
and $\Gamma$ are always separated
by $\Gamma_{N,i}$ to avoid using the trace space $H^{1/2}_{00}(\Gamma)$~\cite{tartar2007introduction}.
Moreover, if $\Gamma_{D,i}$ and $\Gamma$ are not separated,
then the finite element space associated
with the Lagrange multiplier in the mixed formulation
requires modification~\cite{hauret2006bv}.

The strong formulation of the tie contact problem
reads as follows:
\begin{equation}
-\Div\Sig(\U_i) = \F \quad \text{in $\Omega_i$}, \label{eq:tie1}
\end{equation}
where
\begin{equation}
    \Sig(\U_i) = \frac{E}{2(1+\nu)} \Eps(\U_i) + \frac{E \nu}{(1 + \nu)(1 - 2\nu)} \mathrm{tr}\,\Eps(\U_i) \boldsymbol{I}, \quad \Eps(\U_i) = \frac{1}{2}(\nabla \U_i + \nabla \U_i^T), \label{eq:tie2}
\end{equation}
and $(E, \nu)$ are the material parameters
(Young's modulus and Poisson ratio).
The formulation is complemented with
fixed and zero traction
boundary conditions
\begin{equation}
\U_i = \Z \quad \text{on $\Gamma_{D,i}$}, \quad \Sig(\U_i)\N = \Z \quad \text{on $\Gamma_{N,i}$,} \label{eq:tie3}
\end{equation}
and the interface conditions
\begin{equation}
\Lam = \Sig(\U_2)\N = -\Sig(\U_1)\N \quad \text{on $\Gamma$} \label{eq:tie4}
\end{equation}
guaranteeing the continuity of the traction
across the interface $\Gamma$.

Let us define the function space $V_i = \{ \W \in H^1(\Omega_i)^d : \W|_{\Omega_{\Gamma_{D,i}}} = 0\}$
and the dual space $\varLambda = (H^{1/2}(\Gamma)^d)^\prime$
where $d \in \{2,3\}$ is the dimension of the domain
and $H^{1/2}(\Gamma)$ refers to the space of traces of functions in $H^1(\Omega)$
on $\Gamma$.
In the following, we use the notation
$\U = (\U_1, \U_2) \in V_1 \times V_2 = V$
and, consequently, $(\U, \Lam) = (\U_1, \U_2, \Lam) \in V \times \varLambda$.
The variational formulation of \eqref{eq:tie1}--\eqref{eq:tie4} then reads as: find $(\U, \Lam) \in V \times \varLambda$ satisfying
\begin{equation}
\label{eq:varform}
\left\{
\begin{aligned}
  (\Sig(\U_1), \Eps(\V_1))_{\Omega_1} + \langle \Lam, \V_1 \rangle &= (\F, \V_1)_{\Omega_1} \\
  (\Sig(\U_2), \Eps(\V_2))_{\Omega_2} - \langle \Lam, \V_2 \rangle &= (\F, \V_2)_{\Omega_2} \\
  \langle \Mu, \U_1 - \U_2 \rangle &= 0,
\end{aligned}
\right.
\end{equation}
for every $(\V, \Mu) \in V \times \varLambda$.
In \eqref{eq:varform}, $(\cdot,\cdot)_A$ denotes the $L^2$-inner product
over the subdomain $A \subset \Omega$
and $\langle \cdot, \cdot \rangle$ is the $H^{1/2}(\Gamma)^d$-duality pairing.
Equivalently, we write \eqref{eq:varform} as: find $(\U, \Lam) \in V \times \varLambda$ satisfying
\begin{equation}
\Bf(\U, \Lam; \V, \Mu) = \Lf(\V) \quad \forall (\V, \Mu) \in V \times \varLambda, \label{eq:bigvarform}
\end{equation}
where the bilinear form $\Bf : (V \times \varLambda) \times (V \times \varLambda) \rightarrow \mathbb{R}$ and the linear form $\Lf : V \rightarrow \mathbb{R}$ are defined as
\begin{align*}
  \Bf(\W, \X; \V, \Mu) &= \sum_{i=1}^2 (\Sig(\W_i), \Eps(\V_i))_{\Omega_i} + \langle \Mu, \jump{\W} \rangle + \langle \X, \jump{\V} \rangle, \\
  \Lf(\V) &= \sum_{i=1}^2 (\F, \V_i)_{\Omega_i},
\end{align*}
and $\jump{\W} = \W_1 - \W_2$ denotes the jump of the displacement across the interface $\Gamma$.
Existence and uniqueness of \eqref{eq:bigvarform}
are standard and follows the steps
laid out in \cite{babuvska1973finite}.

\section{Stabilized finite element method}

Let $\Th^i$ be a computational mesh of the domain $\Omega_i$, $i=1,2$, with the
mesh parameter $h$ and $\Gh$ be the set of edges/facets of $\Th^1$ on the
interface $\Gamma$.
To introduce the stabilized finite element method,
we define the discrete spaces
\begin{align*}
  V_{h,i} &= \{ \W \in V_i : \W|_T \in \mathcal{P}_k(T)^d~\forall T \in \Th^i \}, \quad V_h = V_{h,1} \times V_{h,2},\\
  \varLambda_h &= \{ \X \in \varLambda : \X|_E \in \mathcal{P}_l(E)^d~\forall E \in \Gh \},
\end{align*}
where $k \geq 1$ and $l \geq 0$ are the polynomial degrees.
We stress that the space $\varLambda_h$ is defined
using the edges/facets of $\Th^1$ only.

Let $h_\Gamma \in L^2(\Gamma)$ be such that $h_\Gamma|_E = h_E~\forall E \in \Gh$, and let $\alpha > 0$ be a stabilization parameter.
The discrete bilinear form $\Bf_h : (V_h \times \varLambda_h) \times (V_h \times \varLambda_h) \rightarrow \mathbb{R}$, defined as
\begin{align}
  \Bf_h(\W_h, \X_h; \V_h, \Mu_h) &= \Bf(\W_h, \X_h; \V_h, \Mu_h) - \alpha (h_\Gamma(\X_h + \Sig(\W_{1,h})\N), \Mu_h + \Sig(\V_{h,1})\N)_\Gamma, \label{eq:stabform}
\end{align}
is continuous and stable in the following mesh-dependent norm (see Theorem~\ref{thm:discstab} below):
\begin{equation}
\enorm{(\W_h, \X_h)}^2_h  = \enorm{\W_h}^2 + \|h_\Gamma^{1/2} \X_h\|_{0,\Gamma}^2. \label{eq:discnorm}
\end{equation}
Note that \eqref{eq:stabform} is stabilized
using the residual of \eqref{eq:tie4}.
In \eqref{eq:discnorm}, the norm for the displacement is given by the strain energy:
\[
\enorm{\W_h}^2 = \enorm{(\W_{1,h}, \W_{2,h})}^2 = \sum_{i=1}^2 (\Sig(\W_{i,h}), \Eps(\W_{i,h}))_{\Omega_i}.
\]

The stabilized finite element formulation now reads as follows: find $(\U_h, \Lam_h) \in V_h \times \varLambda_h$ satisfying
\[
\Bf_h(\U_h, \Lam_h; \V_h, \Mu_h) = \Lf(\V_h) \quad \forall (\V_h, \Mu_h) \in V_h \times \varLambda_h.
\]
The stability of the
formulation is proved
in the following theorem.
\begin{thm}
  \label{thm:discstab}
  For small enough $\alpha > 0$ and any $(\W_h, \X_h) \in V_h \times \varLambda_h$ there exists $C>0$, independent of $h$, such that
  \[
  \Bf_h(\W_h, \X_h; \W_h, -\X_h) \geq C \enorm{(\W_h, \X_h)}^2_h.
  \]
\end{thm}
\begin{proof}
  The definition of the form $\Bf_h$ in \eqref{eq:stabform} and the discrete trace estimate (see, e.g., \cite{gustafsson2020nitsche})
  \[
  C_I \|h_\Gamma^{1/2} \Sig(\W_{1,h})\N \|_{0,\Gamma}^2 \leq (\Sig(\W_{1,h}), \Eps(\W_{1,h}))_{\Omega_1},
  \]
  where $C_I$
  is independent of $h$,
  imply that for any $(\W_h, \X_h) \in V_h \times \varLambda_h$ it holds
  \begin{align*}
    \Bf_h(\W_h, \X_h; \W_h, -\X_h) &= \enorm{(\W_h, \X_h)}^2_h - \alpha \|h_\Gamma^{1/2} \Sig(\W_{1,h})\N \|_{0,E}^2 \\
    &\geq \left(1 - \frac{\alpha}{C_I}\right) \enorm{\W_h}^2 + \alpha \|h_\Gamma^{1/2} \X_h \|_{0,\Gamma}^2.
  \end{align*}
  The terms on the right hand side remain positive if $0 < \alpha < C_I$.
\end{proof}

\begin{rem}
From the proof of Theorem~\ref{thm:discstab}, it is obvious that the result holds for any
$L^2$-conforming Lagrange multiplier
space, continuous or discontinuous.
One may even use non-local functions to define a global spectral basis.
This idea has been used, e.g., in~Hansbo et al.~\cite{hansbo2005nitsche, hansbo2005lagrange}
to simplify the implementation and circumvent
the need for a conforming
integration mesh.
Moreover, it is possible to eliminate a discontinuous, piecewise polynomial
Lagrange multiplier locally on each element to
arrive at a positive definite Nitsche-type method~\cite{gustafsson2017mixed, gustafsson2019error, gustafsson2020nitsche}.
\end{rem}

As a consequence of Theorem~\ref{thm:discstab}, we have the following
quasi-optimality estimate:
\begin{thm}
\label{thm:quasi} Suppose $\U \in H^{3/2}(\Omega_1)^d \times H^{3/2}(\Omega_2)^d$. There exists $C>0$ such that
   \begin{align*}
     &\enorm{(\U - \U_{h}, \Lam - \Lam_h)}_h \leq C \inf_{\substack{(\V_h, \Mu_h) \\ \in V_h \times \varLambda_h}} \enorm{(\U - \V_h, \Lam - \Mu_h)}_h.
   \end{align*}
\end{thm}
\begin{proof}
Discrete stability (Theorem~\ref{thm:discstab}), the orthogonality result
\[
\Bf_h(\U - \U_h, \Lam - \Lam_h; \V_h, \Mu_h) = 0 \quad \forall (\V_h, \Mu_h) \in V_h \times \varLambda_h,
\]
and continuity of $\Bf_h$,
imply that
\begin{align*}
&\enorm{(\U_{h} - \V_{h}, \Lam_h - \Mu_h)}_h^2 \\
&\leq C \Bf_h(\U_{h} - \V_{h}, \Lam_h - \Mu_h; \U_{h} - \V_{h}, \Lam_h - \Mu_h) \\
&= C \Bf_h(\U - \V_h, \Lam - \Mu_h; \U_h - \V_h, \Lam_h - \Mu_h) \\
&\leq C\|(\U - \V_h, \Lam - \Mu_h)\|_h \|(\U_h - \V_h, \Lam_h - \Mu_h)\|_h.
\end{align*}
Dividing both sides by $\|(\U_h - \V_h, \Lam_h - \Mu_h)\|_h$
and using the triangle inequality leads to the proof.
\end{proof}

\begin{rem}
It is possible to prove Theorem~\ref{thm:discstab} in the $V$-norm for the displacement and in the continuous $H^{-1/2}(\Gamma)$-norm for the Lagrange multiplier,
to arrive at a similar estimate without
extra regularity assumptions~\cite{gustafsson2017mixed}.
\end{rem}

\section{Mixed finite element method}

Let us next consider a mixed method
using a continuous
Lagrange multiplier, i.e.~the spaces
\begin{equation}
  \label{eq:clagmult}
  \begin{aligned}
        V_{h,i} &= \{ \W \in V_i : \W|_T \in \mathcal{P}_k(T)^d~\forall T \in \Th^i \}, \quad V_h = V_{h,1} \times V_{h,2},\\
      \varLambda_h^c &= \{ \X \in \varLambda \cap C(\Gamma)^d : \X|_E \in \mathcal{P}_l(E)^d~\forall E \in \Gh \},
      \end{aligned}
\end{equation}
where $k, l \geq 1$.
This approach is mentioned in the early work
of Pitkäranta~\cite{pitkaranta1979boundary}
where it is noted that
the discrete stability holds
if the space of Lagrange multipliers
is a subset of the trace space of $V_{h,1}$.
However, the
stability may not, in general, be
uniform in $h$
without
additional restrictions on the
finite element mesh.

The discrete mixed formulation reads as follows:
find $(\U_h, \Lam_h) \in V_h \times \varLambda_h^c$ satisfying
\[
\Bf(\U_h, \Lam_h; \V_h, \Mu_h) = \Lf(\V_h) \quad \forall (\V_h, \Mu_h) \in V_h \times \varLambda_h^c.
\]
The Babu\v{s}ka--Brezzi condition, guaranteeing stability is written as: there exists
$C>0$, independent of $h$, such that
\begin{equation}
  \label{eq:bb}
  \sup_{\V_h \in V_h} \frac{\langle \jump{\V_h}, \X_h \rangle}{\enorm{\V_h}} \geq C \| \X_h \|_{-\frac12} \quad \forall \X_h \in \varLambda_h^c.
\end{equation}
Here we establish the uniform stability for the mixed method
using the mesh-dependent norm in $\varLambda_h^c$ and the
following auxiliary result~\cite{gustafsson2017mixed}.
\begin{lem}
  If there exists $C > 0$, independent of $h$, such that
  \begin{equation}
    \label{eq:dbb}
      \sup_{\V_h \in V_h} \frac{\langle \jump{\V_h}, \X_h \rangle}{\enorm{\V_h}} \geq C \| h_\Gamma^{1/2} \X_h \|_{0,\Gamma}^2 \quad \forall \X_h \in \varLambda_h^c,
  \end{equation}
  then the Babu\v{s}ka--Brezzi condition \eqref{eq:bb} holds. \label{lem:step}
\end{lem}
Lemma~\ref{lem:step} can be proven independently of the finite element spaces for any shape regular mesh families using the stability of the continuous problem and Cl\'{e}ment interpolation, cf.~\cite{gustafsson2017mixed}.
The following assumption on the finite element mesh $\Gh$
is sufficient for stability
in the mesh-dependent norm,
see \cite{bank2014h} for
a discussion on similar estimates with graded meshes.

\begin{ass}
\label{ass:projstab}
  The $L^2$-projection
  $\pi_h : L^2(\Gamma)^d \rightarrow \varLambda_h^c$
  satisfies
  \[
    \| h^{-1/2}_\Gamma \pi_h \bs{v} \|_{0,\Gamma} \leq C \| h^{-1/2}_\Gamma \bs{v} \|_{0,\Gamma} \quad \forall \V \in L^2(\Gamma)^d.
  \]
\end{ass}

\begin{thm}
  If $k=l=1$ in \eqref{eq:clagmult}
  and the finite element mesh satisfies Assumption~\ref{ass:projstab}
  then the stability condition \eqref{eq:dbb} holds.\label{thm:mixstab}
\end{thm}
\begin{proof}
 Let
 $\mathcal{E}_h \X_h$,
 where
 $\mathcal{E}_h : \varLambda_h^c \rightarrow V_{h,1}$,
 be the extension of a piecewise
 linear function $\X_h$ on $\Gamma$ to $\Omega_1$ defined via the finite element
 basis so that the nodes outside of the interface $\Gamma$ have value zero.
 For such an extension it holds \cite{verfurth2013}
 \begin{equation}
   \| \nabla \mathcal{E}_h \Mu_h \|_{0,\Omega_1}^2 \leq C_E \| h_\Gamma^{-1/2} \Mu_h \|_{0,\Gamma}^2 \quad \forall \Mu_h \in \varLambda_h^c. \label{eq:mixpf1}
 \end{equation}
 From the definition of the $L^2$-projection, choosing $\W_h = (\W_{1,h}, 0) = (\mathcal{E}_h \pi_h(h_\Gamma \X_h), 0)$,
 it follows that
 \begin{equation}
    \langle \jump{\W_h}, \X_h \rangle = \langle \pi_h(h_\Gamma \X_h), \X_h \rangle = \| h_\Gamma^{1/2} \X_h \|_{0,\Gamma}^2. \label{eq:mixpf2}
 \end{equation}
 Using Korn's inequality, property \eqref{eq:mixpf1} and Assumption~\ref{ass:projstab},
 we obtain
 \begin{equation}
 \label{eq:mixpf3}
 \begin{aligned}
    \enorm{\W_h}^2 &= (\Sig(\W_{1,h}), \Eps(\W_{1,h})_{\Omega_1} \\
    &\leq C_K \| \nabla \mathcal{E}_h \pi_h(h_\Gamma \X_h) \|_{0,\Omega_1}^2 \\
    &\leq C_K C_E \| h_\Gamma^{-1/2} \pi_h(h_\Gamma \X_h) \|_{0,\Gamma}^2 \\
    &\leq C_K C_E C \| h_\Gamma^{1/2} \X_h \|_{0,\Gamma}^2.
 \end{aligned}
 \end{equation}
 The result now follows from \eqref{eq:mixpf2} and \eqref{eq:mixpf3}.
\end{proof}

\begin{rem}
As a stronger alternative to
Assumption~\ref{ass:projstab},
we can assume that
there exist $C_1, C_2 > 0$,
independent of $h$, satisfying
\begin{equation}
\label{eq:assumption2}
 C_1 \min_{F \in \Gh} h_F \leq h_E \leq C_2 \max_{F \in \Gh} h_F \quad \forall E \in \Gh.
\end{equation}
Unfortunately, \eqref{eq:assumption2} requires the mesh to be uniform.
\end{rem}

\begin{rem}
For $P_{k} - P_{k-1}$, $k\geq 2$,
the proof of Theorem~\ref{thm:mixstab}
can be adapted, see also
 \cite{seshaiyer2000uniform}
 where similar methods are discussed for multiple subdomains.
 For $P_{k} - P_{k-2}$, $k\geq 2$,
 with a discontinuous Lagrange multiplier, the proof simplifies for $d=2$
 because we can use the edge bubbles of $P_k$ on $\Gamma$.
\end{rem}

\section{Numerical experiments}

In this section, we perform experiments using
both methods and different polynomial orders.
We refer to the method using polynomial order $k$
for the displacement and $l$ for the continuous Lagrange
multiplier as $P_k - P_l$ mixed method.
The two-dimensional examples are implemented using scikit-fem~\cite{skfem2020}
and visualized using matplotlib~\cite{hunter2007matplotlib}
whereas the three-dimensional examples are implemented using the finite element software Elmer~\cite{malinen2013elmer}
and visualized using Paraview~\cite{ahrens2005paraview}.
In all examples
the material parameters are taken to be $E=10^3$ and
$\nu=0.3$ in both bodies.
The source code and problem definition files for
reproducing the numerical examples are
available in~\cite{gustafsson_tom_2022}.

\subsection{Square against square, unstable example}

For completeness, we first demonstrate
the instability of $P_1-P_0$
mixed method and how stabilization recovers
the expected rate of convergence.
The domain is
\[
    \Omega_1 \cup \Omega_2 = (0, 1)^2 \cup (1, 1.5) \times (0, 1).
\]
The interface is
\[
    \Gamma = \{ (1, y) : 0 < y < 1 \}
\]
while the boundaries are
\[
    \Gamma_{D,1} = \{ (0, y) : 0 < y < 1 \}, \quad
    \Gamma_{D,2} = \{ (1.5, y) : 0 < y < 1 \}, \quad
    \Gamma_{N,i} = \partial \Omega_i \setminus (\Gamma_{D,i} \cup \Gamma),
\]
and the essential boundary conditions are given by
\[
   \U_1|_{\Gamma_{D,1}} = (0.1, 0), \quad \U_2|_{\Gamma_{D,2}} = (0, 0).
\]
First, we demonstrate that the instability of $P_1 - P_0$ mixed method with a piecewise
constant Lagrange multiplier is clearly visible
when using matching vertices.
The initial meshes are given in Figure~\ref{fig:mesh1}.
The Lagrange multipliers are visualized
in Figure~\ref{fig:lagmult} from which it
becomes clear that the tangential Lagrange multiplier
does not properly represent the true solution.
This is further exemplified in Figure~\ref{fig:p1} (left)
where we observe that the convergence rate of the Lagrange multiplier
is seriously affected
by the lack of stabilization.

For the stabilized method, we
observe superconvergence of the Lagrange multiplier with the error
of the order $\mathcal{O}(h^{3/2})$ --
instead of $\mathcal{O}(h)$ as suggested by the linear
approximation of the primal variable.
The reference solution was calculated using
a stabilized $P_3-P_2$ method and a very fine mesh.

\subsection{Square against square, $P_1-P_1$ methods}

\label{sec:squaresquare}

In order to demonstrate the stability of the mixed $P_1-P_1$ method,
we solve the same problem as above using both the mixed method and
its stabilized counterpart using a continuous Lagrange multiplier.
A nonmatching mesh is used
to demonstrate the approach in the case of general meshes.
The convergence rates are in
Figure~\ref{fig:p1} (right) and the Lagrange multipliers
are visualized in Figure~\ref{fig:clagmult}.
As a conclusion, both methods give essentially the
same results. As above, we observe superconvergence of the
Lagrange multiplier with the error of order $\mathcal{O}(h^2)$.

\subsection{Three-dimensional examples}

The mixed $P_1-P_1$ method is implemented in Elmer~\cite{malinen2013elmer}
for three-dimensional geometry.
We first present an example where a cylinder with the radius and height of 0.5 and 1, respectively,
is pushed against a cube with side length of 1, see Figure~\ref{fig:meshes1}.
The external force is introduced through inhomogeneous Dirichlet boundary condition, i.e.~the
bottom surface of the block is fixed whereas the top surface of the cylinder is
forced downwards 0.1 units.
Consequently, the exact solution is independent of the cylinder's axial rotation
and the discrete problem can be solved
using several angles to observe how the
orientation of elements affects
the contact force.

The cylinder is discretized using linear tetrahedral elements in Gmsh~\cite{geuzaine2009gmsh} and
the block is discretized using trilinear hexahedral tensor product grid.
The normal and tangential components of the contact force are given for three different
values of cylinder's axial rotation
(0, $\pi/6$ and $\pi/3$)
in Figures \ref{fig:normals1} and \ref{fig:tangents1}.
All rotation angles lead to practically
identical normal and tangential contact forces
which suggests that the method
is robust with respect to the shape
and the orientation of the elements
at the contact interface.


The second three-dimensional example concerns a
more complicated
geometry with a cylindrical shaft
cutting through a block.
The cylinder
has radius $0.25$ and height 1,
and the block is a cube with side length of 1, see
Figure~\ref{fig:meshes3}.
The block is squeezed 0.1 units using a similar boundary condition as before.
The cylinder is rotated around
its axis and
results are obtained for several angles of rotation.
The contact forces 
are given in Figures~\ref{fig:normals3} and \ref{fig:tangents3} demonstrating
once again the independence of the
results on the finite element mesh.

The final example is a
three-dimensional version
of the convergence study in
Section~\ref{sec:squaresquare}.
More precisely,
a cubical block with side length of 1 is pushed downward against another block
with same horizontal dimensions and of height 0.5.
As we do not possess 
a sufficiently accurate
reference solution,
we compute the strain energy
at different mesh refinement
levels and visualize
its change as a
function of the mesh parameter,
see Figure~\ref{fig:strainconv}
where a linear rate is
demonstrated.
The corresponding contact forces 
are shown in Figures~\ref{fig:normloads} and \ref{fig:sliploads}.

\section{Discussion}

Mortar finite element methods and tie contact problems are
naturally presented in
a mixed variational formulation
and their discretizations often use 
special finite element spaces.
In this paper,
we have presented two finite element formulations,
mixed and stabilized,
for the approximation of the elastic tie contact problem.
Both use standard finite element spaces
and are stable for continuous $P_1 - P_1$
elements, the mixed method
under certain mild assumptions on the
boundary conditions and the finite
element mesh, and lead
to similar numerical results.
The mixed method has been implemented
for general two- and three-dimensional contact problems in Elmer~\cite{malinen2013elmer}.


\section*{Acknowledgements}

The work was supported by the Academy of Finland (Decisions 324611 and 338341)
and the Portuguese government through FCT (Funda\c{c}\~ao para a Ci\^encia e a Tecnologia), I.P., under the projects PTDC/MAT-PUR/28686/2017 and UIDB/04459/2020.



\begin{figure}
    \centering
    \includegraphics[width=0.4\textwidth]{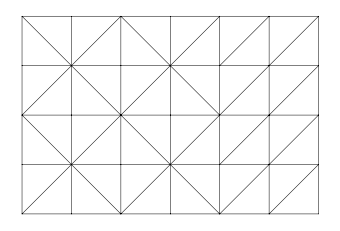}
    \includegraphics[width=0.37\textwidth]{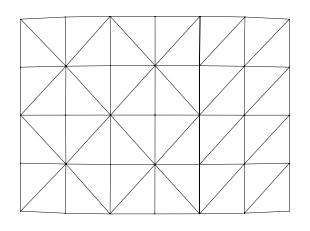}
    \caption{The initial meshes (left) and the deformed meshes (right)
    with matching nodes for the first numerical experiment.}
    \label{fig:mesh1}
\end{figure}

\begin{figure}
    \centering
    \includegraphics[width=0.49\textwidth]{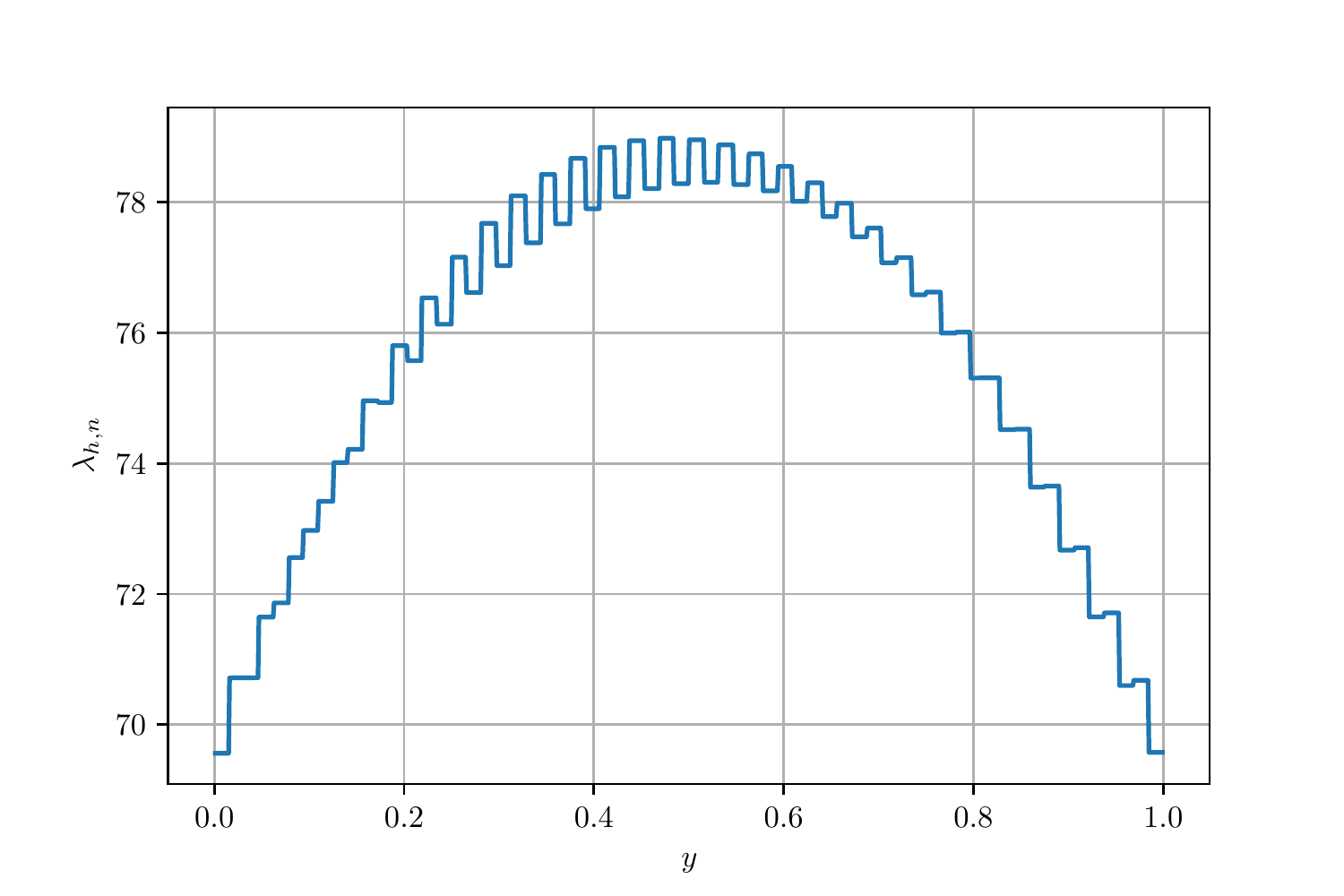}
    \includegraphics[width=0.49\textwidth]{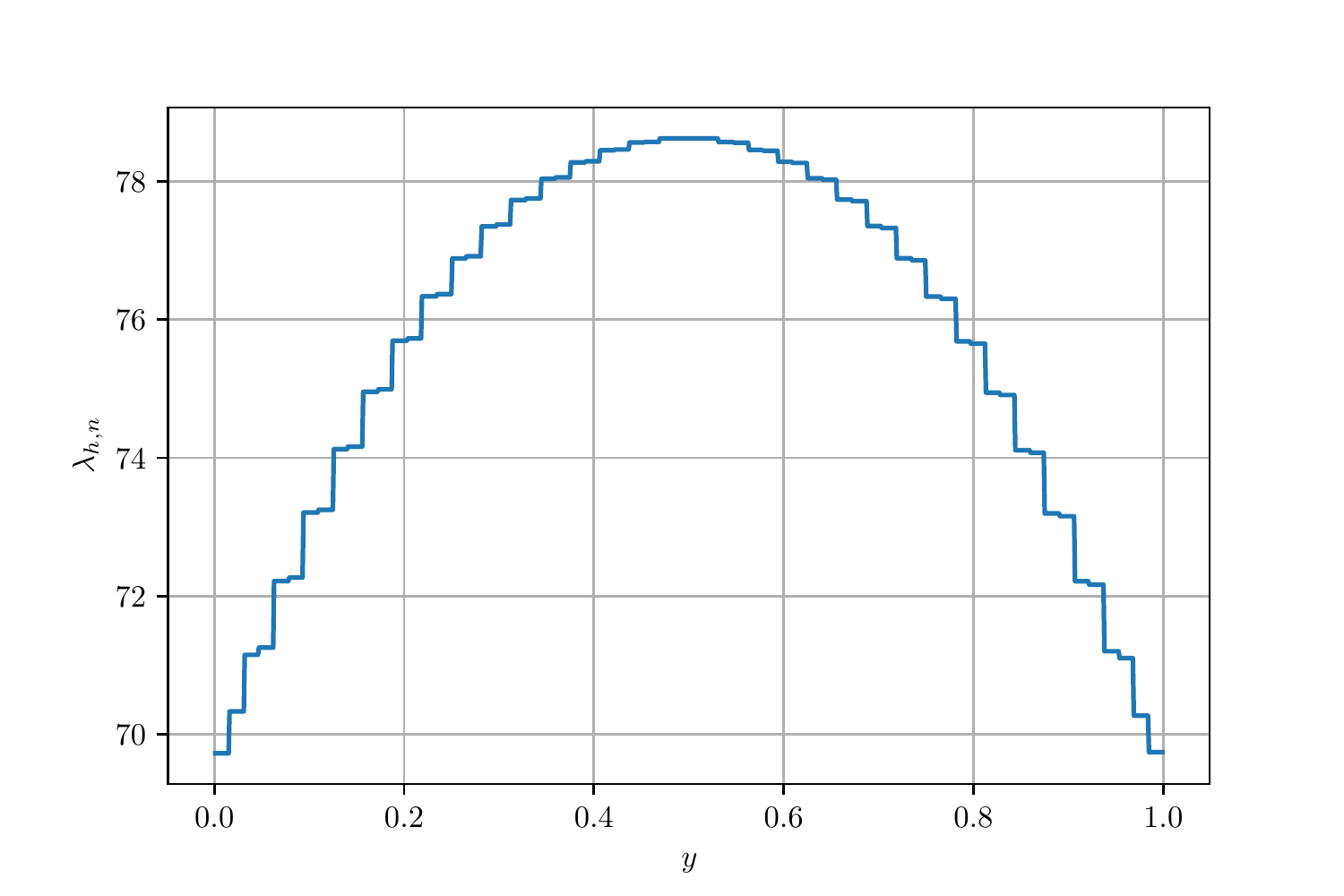}\\
    \includegraphics[width=0.49\textwidth]{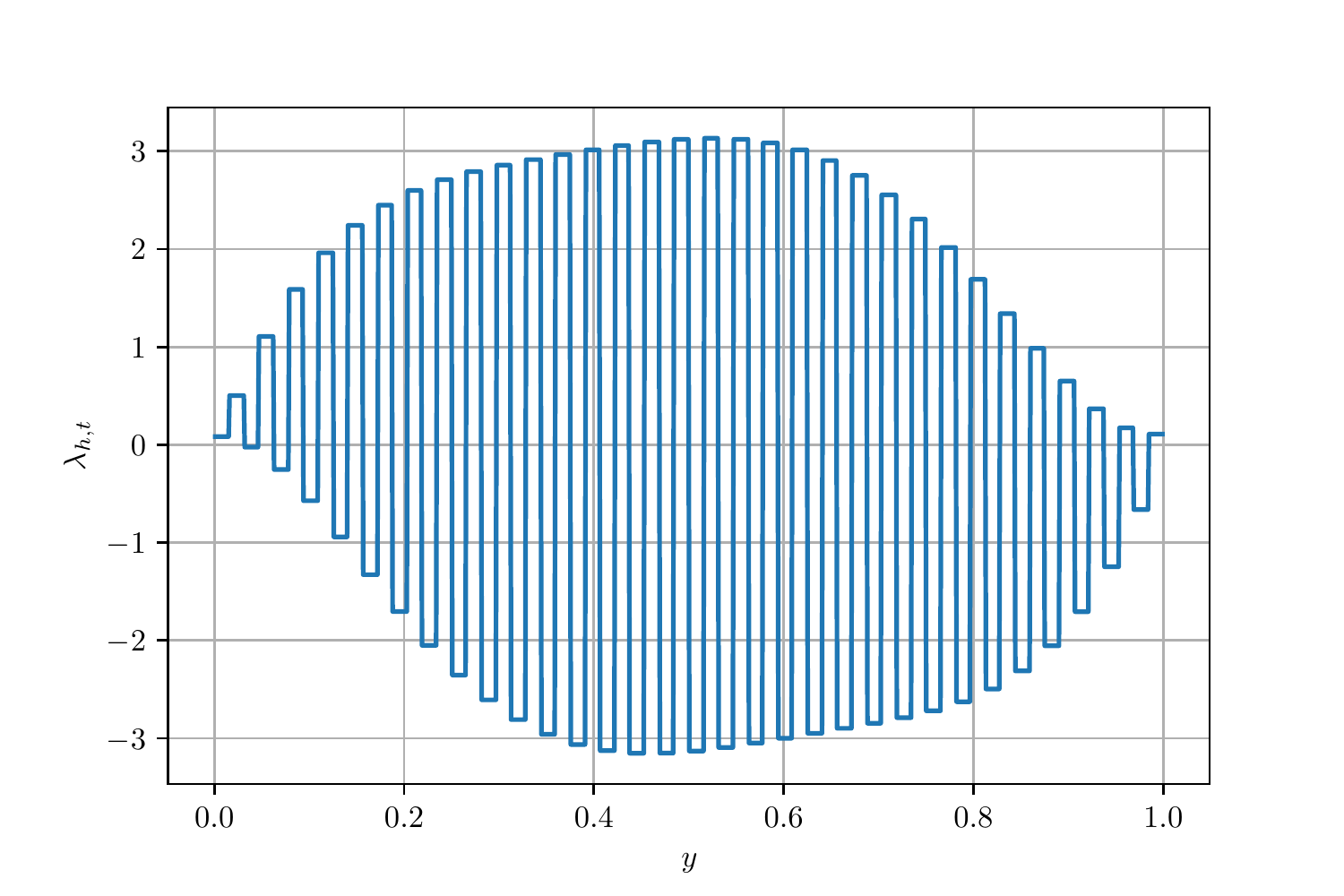}
    \includegraphics[width=0.49\textwidth]{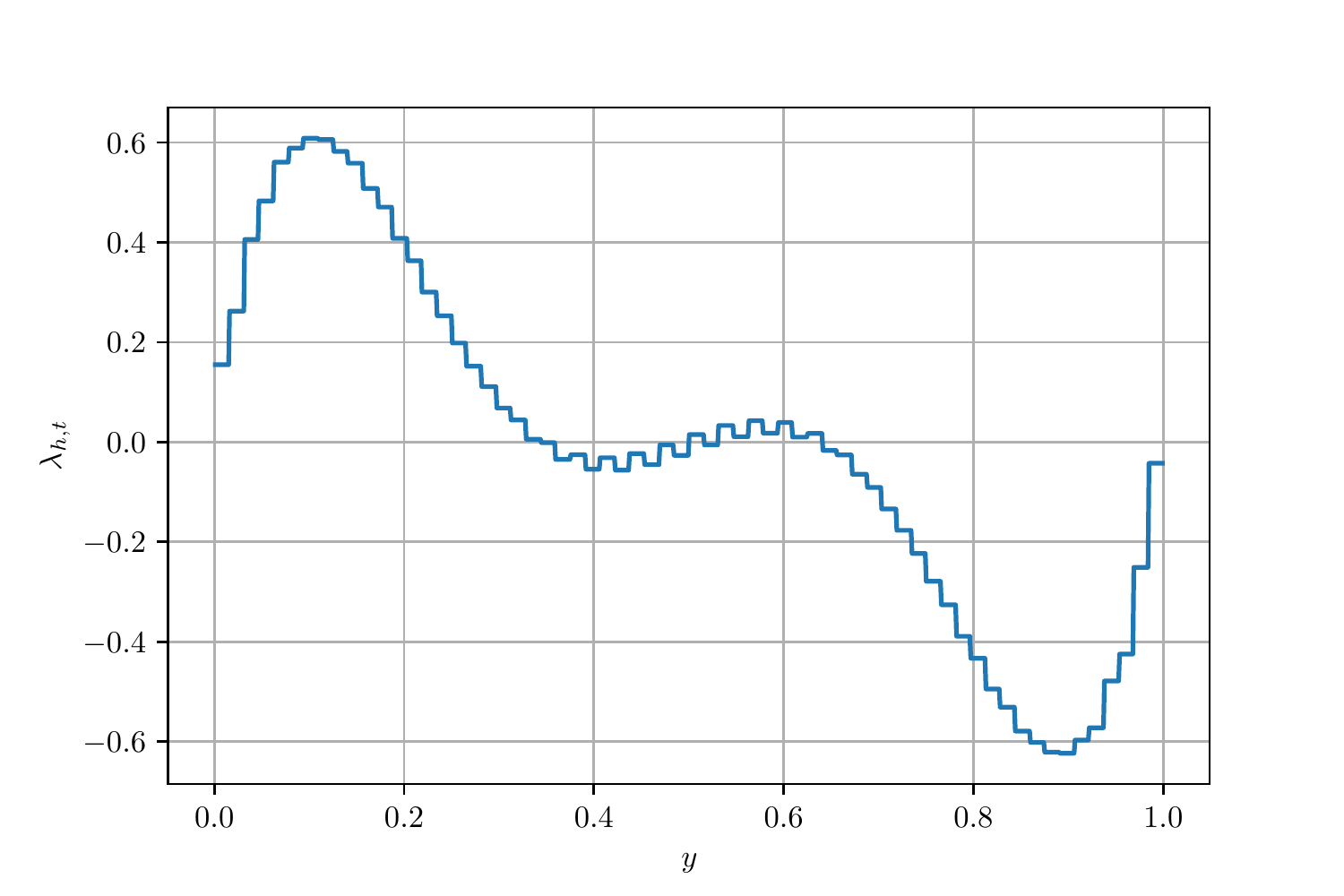}
    \caption{A comparison of the Lagrange multipliers for the unstable $P_1 - P_0$ mixed method (left) and its stabilized counterpart (right).}
    \label{fig:lagmult}
\end{figure}

\begin{figure}
    \centering
    \includegraphics[width=0.51\textwidth]{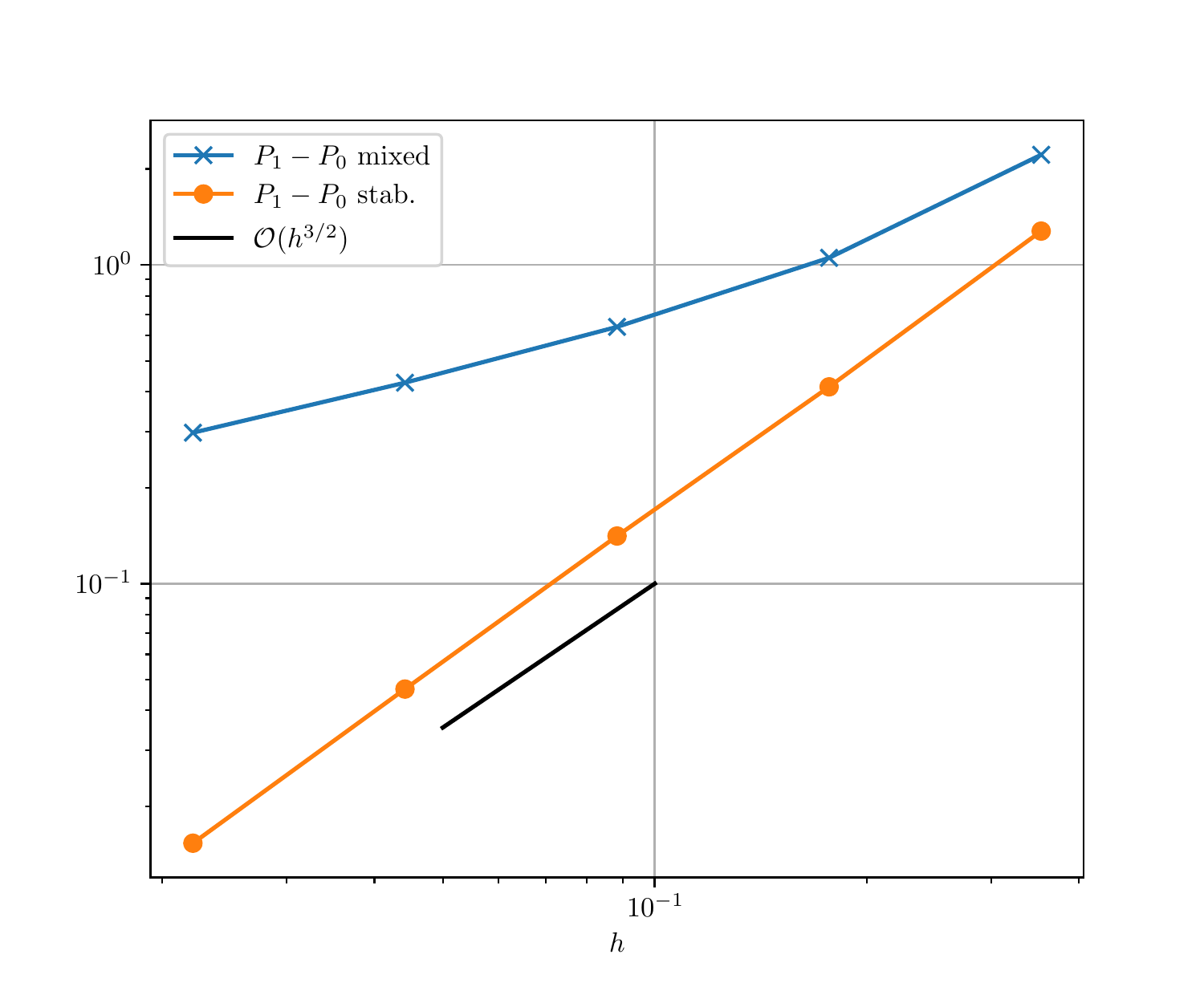}
    \hspace{-0.7cm}
    \includegraphics[width=0.51\textwidth]{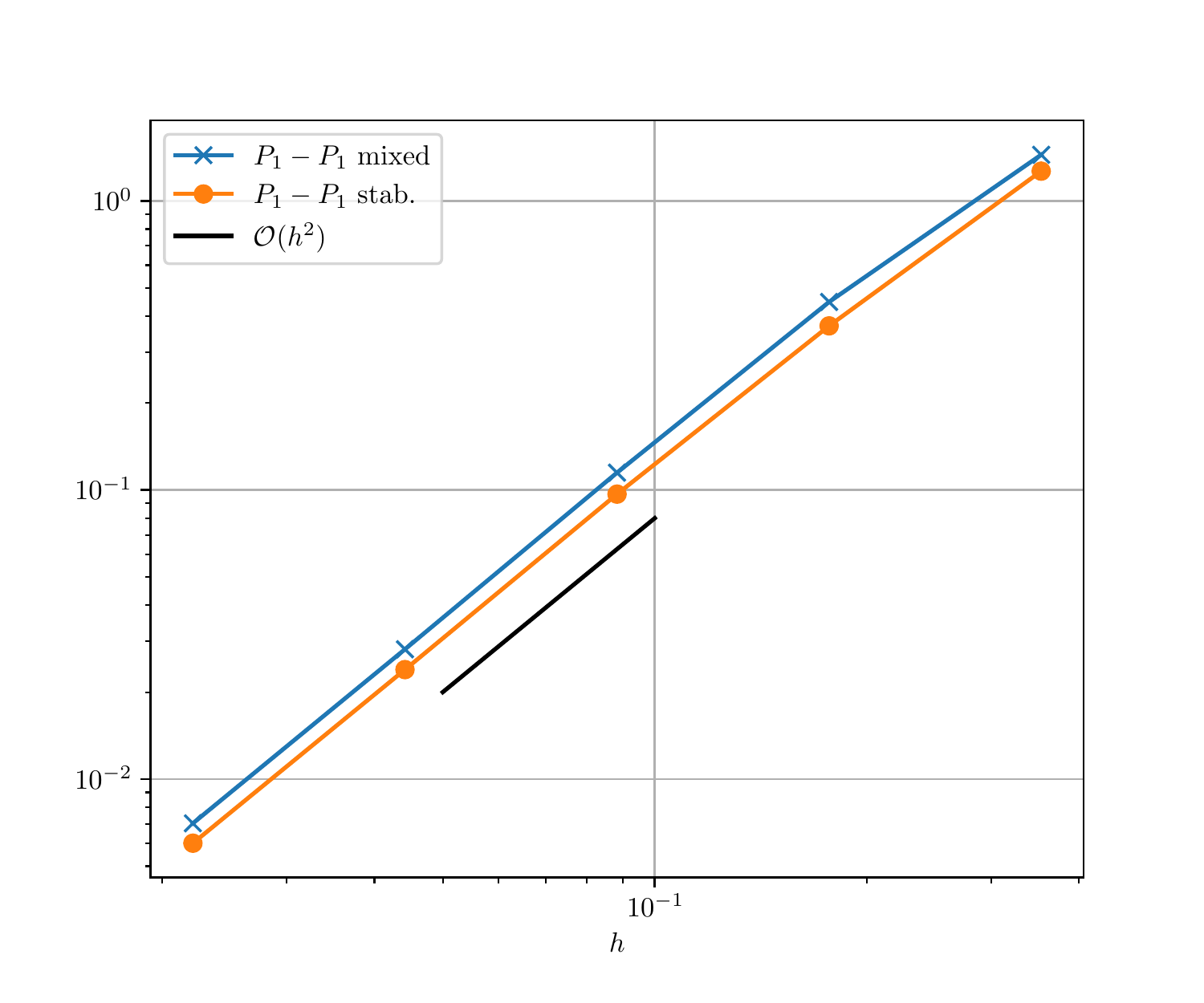}
    \caption{The error of the Lagrange multiplier for $P_1 - P_0$ methods (left) and $P_1 - P_1$ methods (right) in the discrete $H^{-1/2}$-norm.}
    \label{fig:p1}
\end{figure}

\begin{figure}
    \centering
    \includegraphics[width=0.4\textwidth]{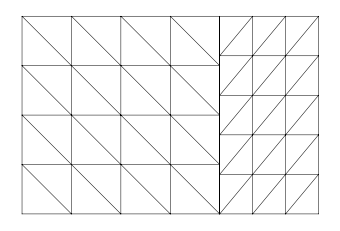}
    \includegraphics[width=0.37\textwidth]{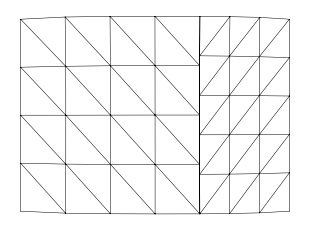}
    \caption{The initial meshes (left) and the deformed meshes (right)
    with nonmatching nodes for the second numerical experiment.}
    \label{fig:mesh2}
\end{figure}

\begin{figure}
    \centering
    \includegraphics[width=0.49\textwidth]{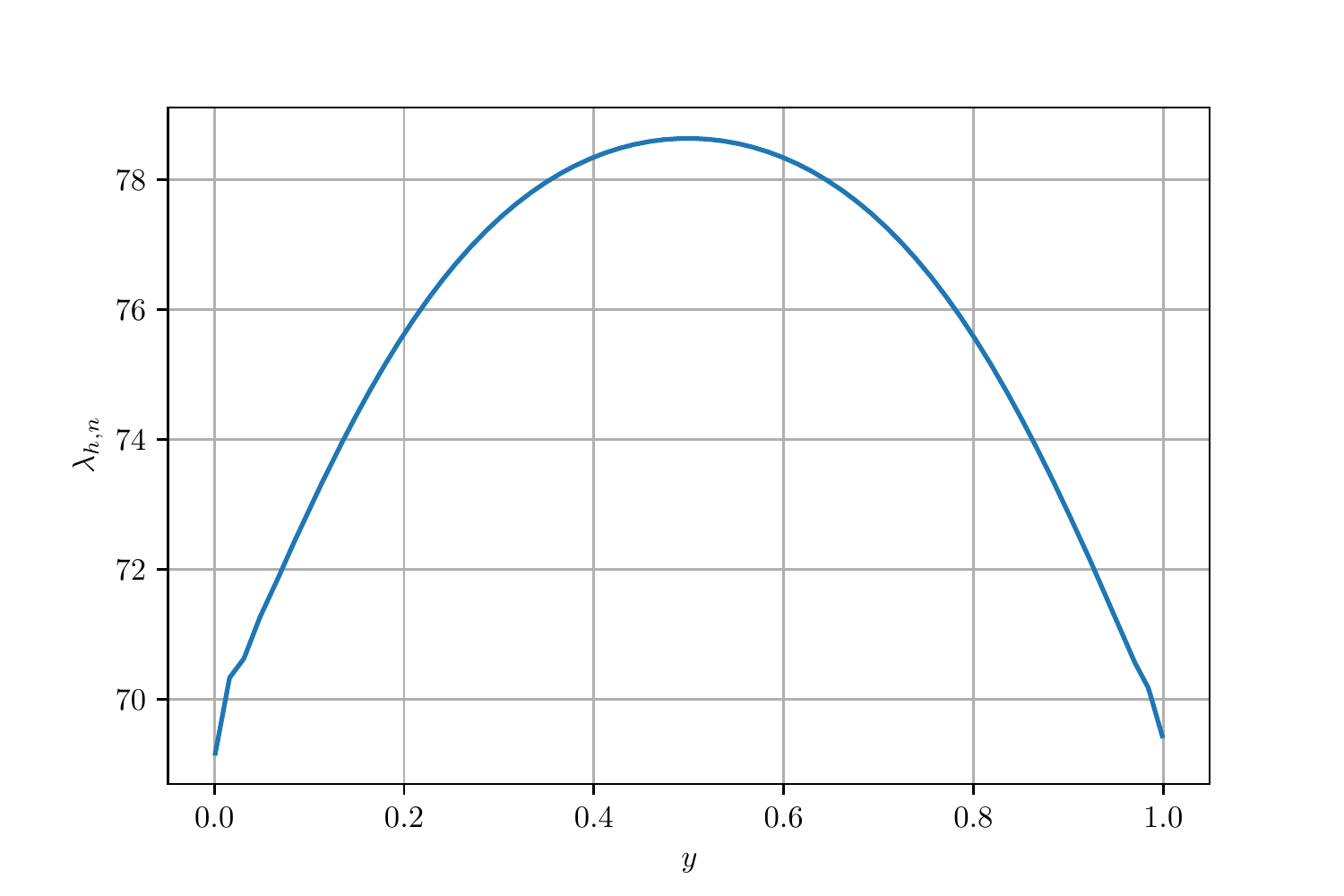}
    \includegraphics[width=0.49\textwidth]{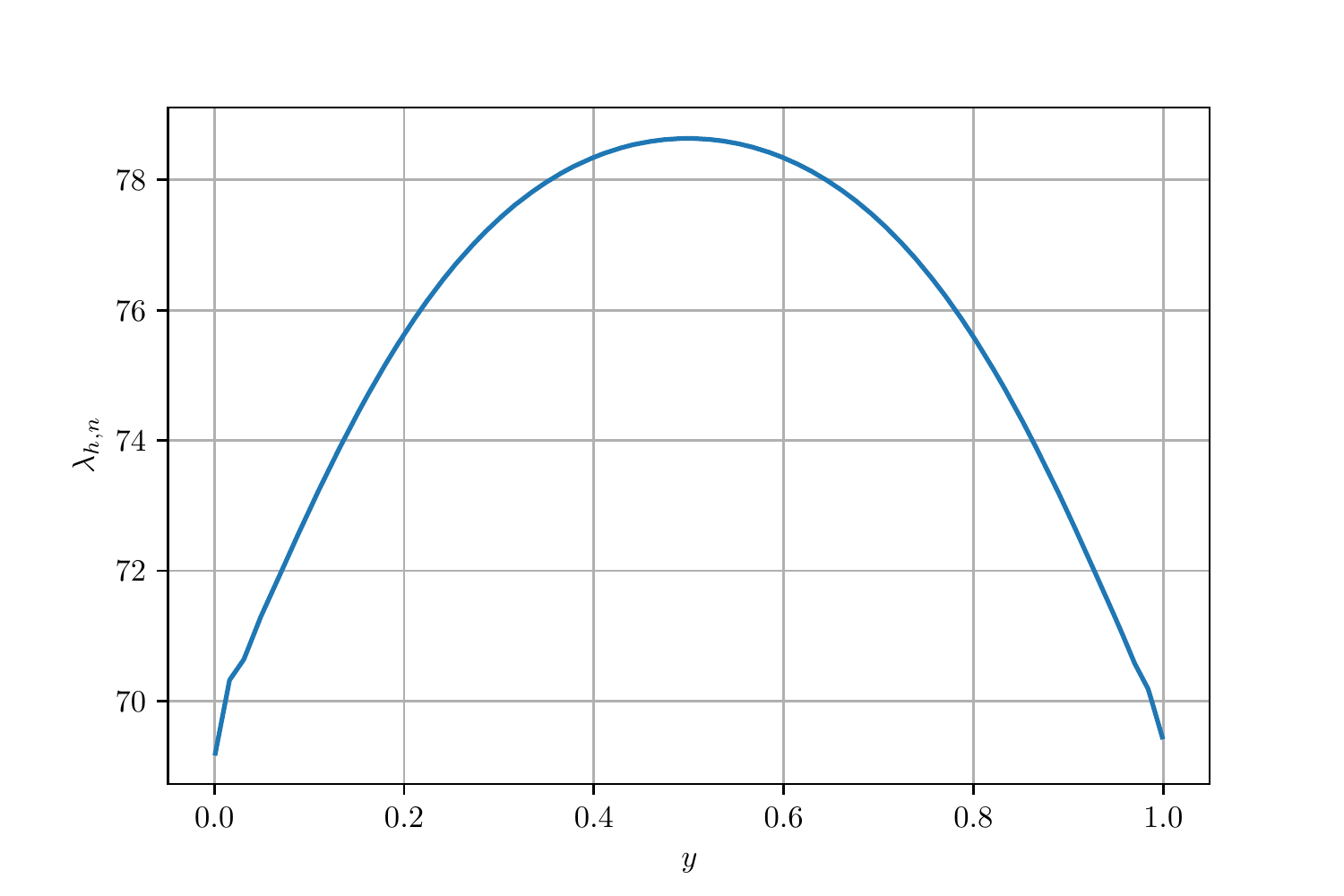}\\
    \includegraphics[width=0.49\textwidth]{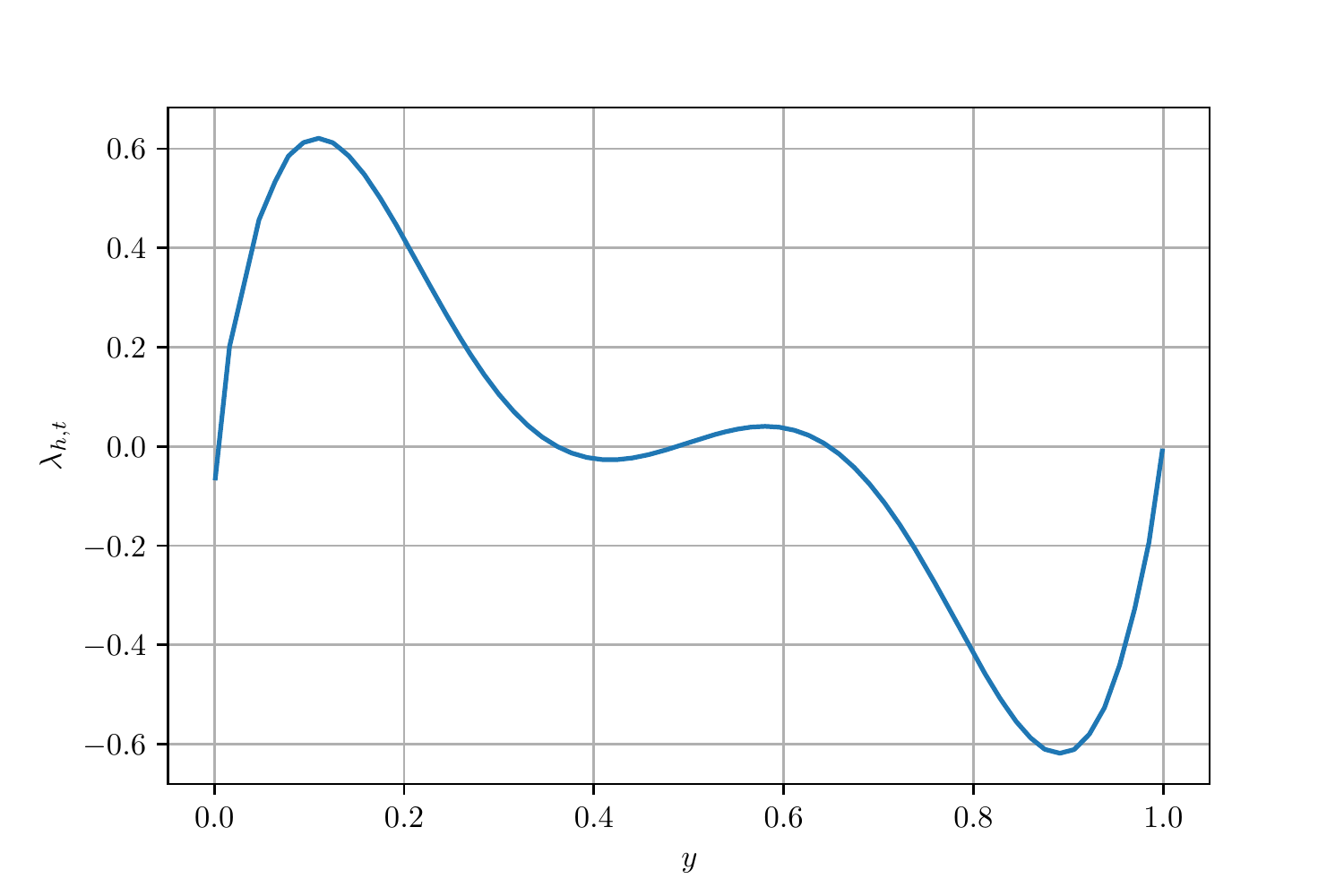}
    \includegraphics[width=0.49\textwidth]{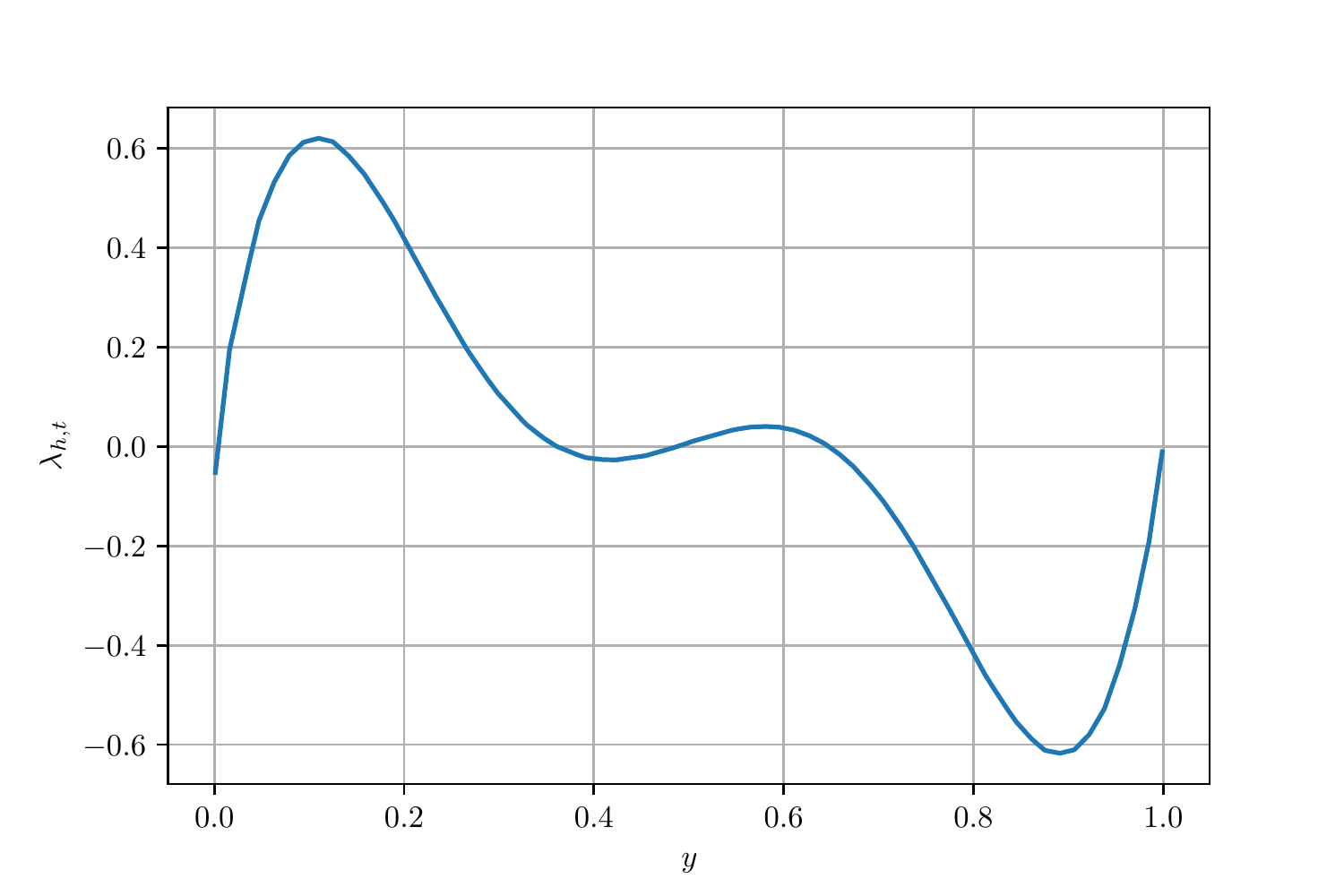}
    \caption{A comparison of the Lagrange multipliers for the mixed (left) and stabilized (right) $P_1 - P_1$ methods.}
    \label{fig:clagmult}
\end{figure}

\FloatBarrier

\begin{figure}
    \centering
    \includegraphics[width=0.25\textwidth]{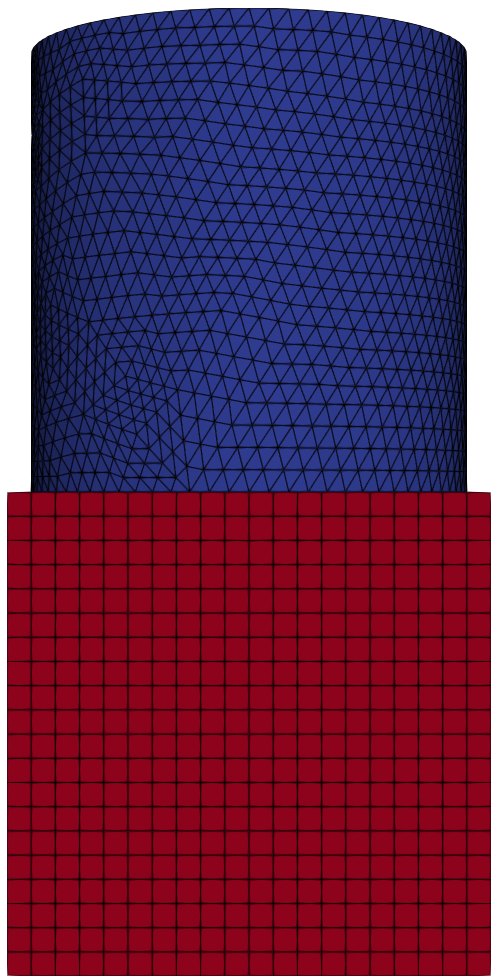}
    \includegraphics[width=0.30\textwidth]{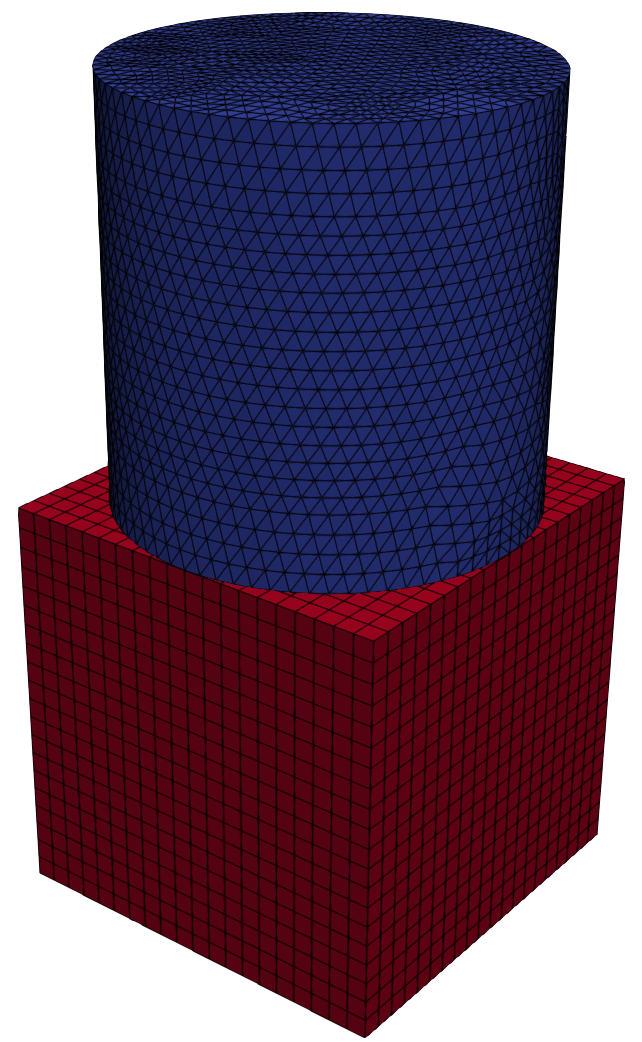}
    \caption{Cylinder and block geometry and computational mesh.}
    \label{fig:meshes1}
\end{figure}

\begin{figure}
    \centering
    \includegraphics[width=\textwidth]{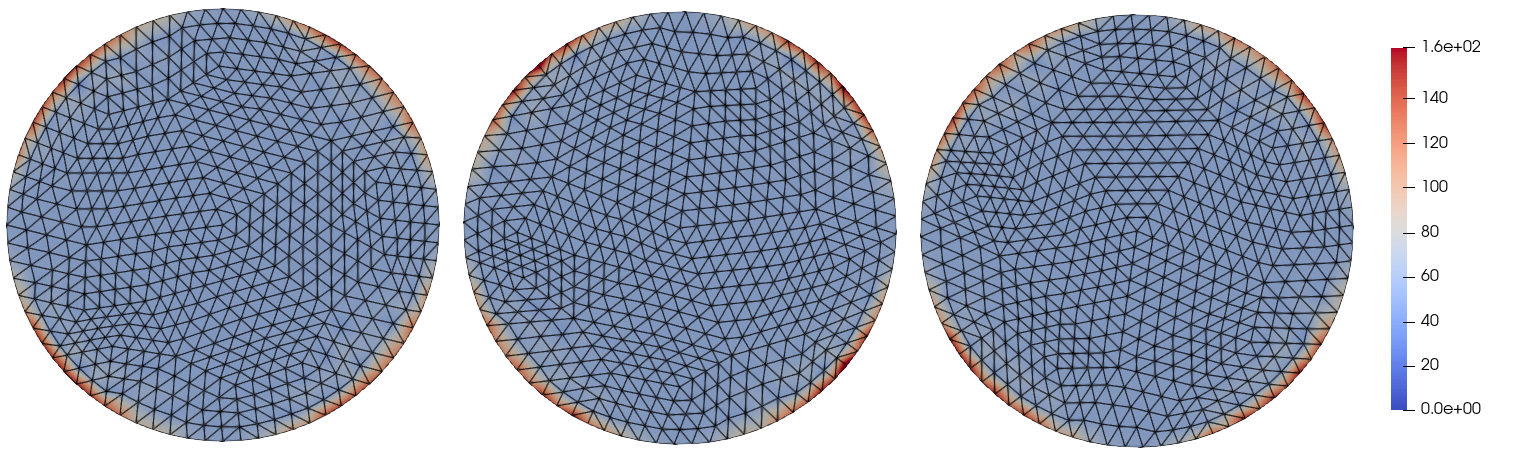}
    \caption{The normal contact loads for three different angles.}
    \label{fig:normals1}
\end{figure}

\begin{figure}
    \centering
    \includegraphics[width=\textwidth]{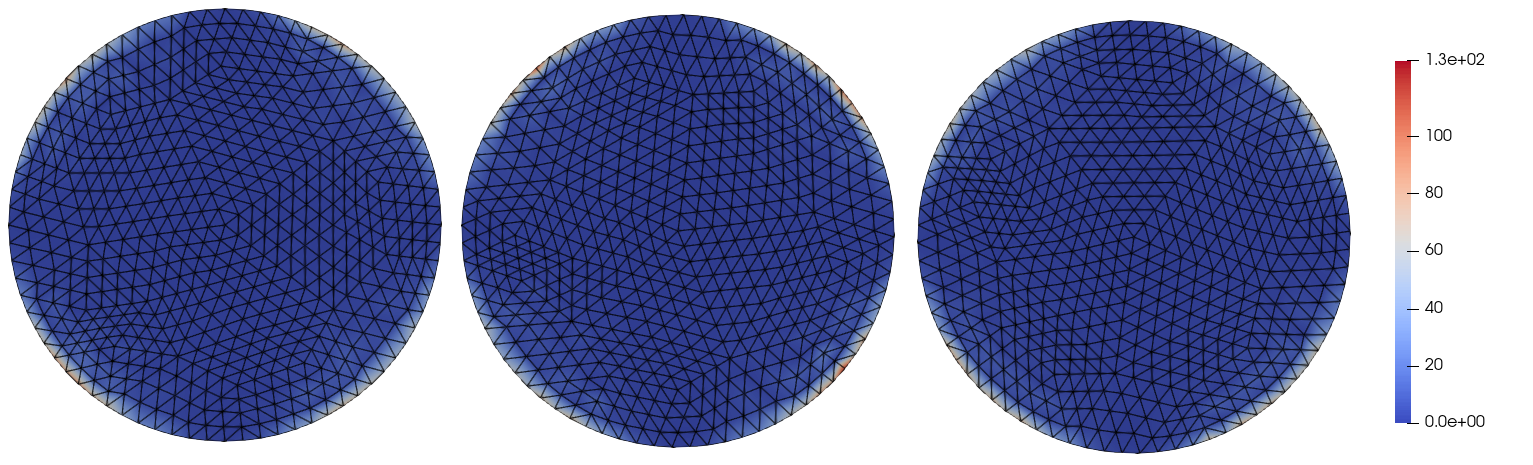}
    \caption{The tangential contact load magnitudes for three different angles.}
    \label{fig:tangents1}
\end{figure}




\begin{figure}
    \centering
    \includegraphics[width=0.4\textwidth]{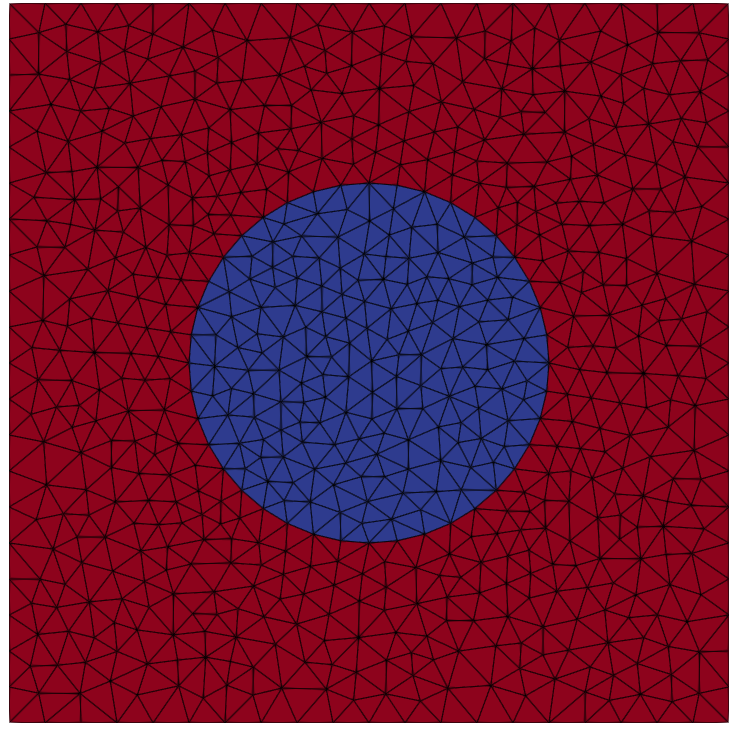}
    \includegraphics[width=0.4\textwidth]{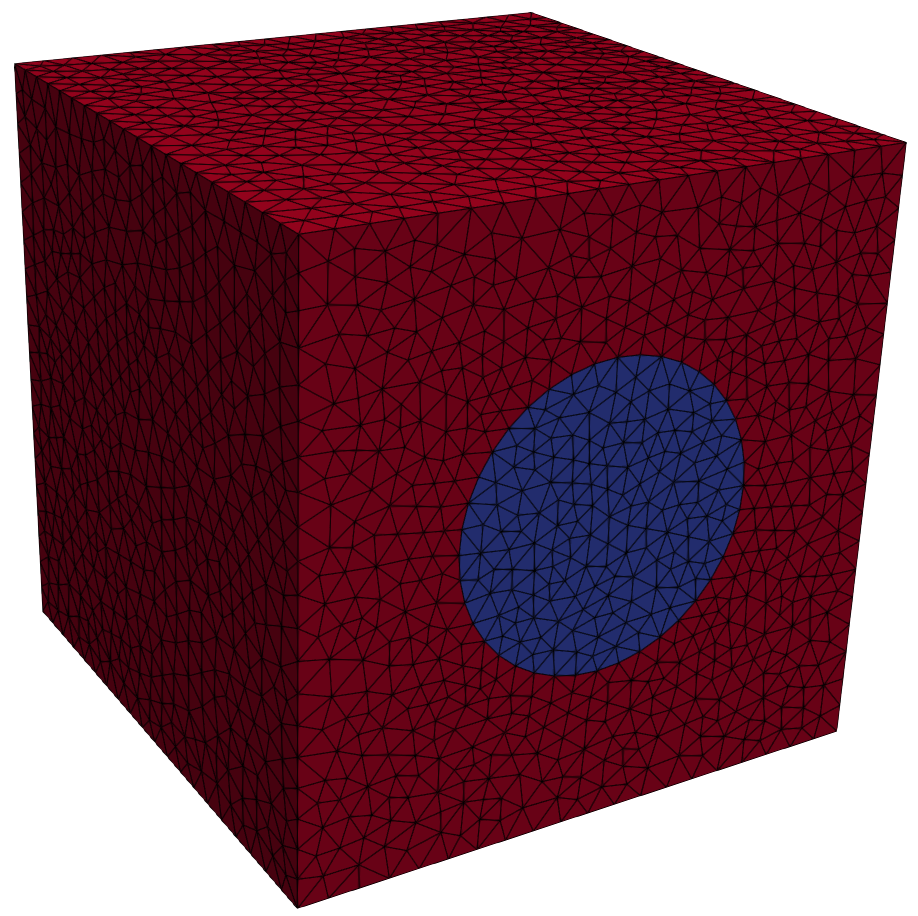}
    \caption{Cylinder cutting the cube.}
    \label{fig:meshes3}
\end{figure}

\begin{figure}
    \centering
    \includegraphics[width=\textwidth]{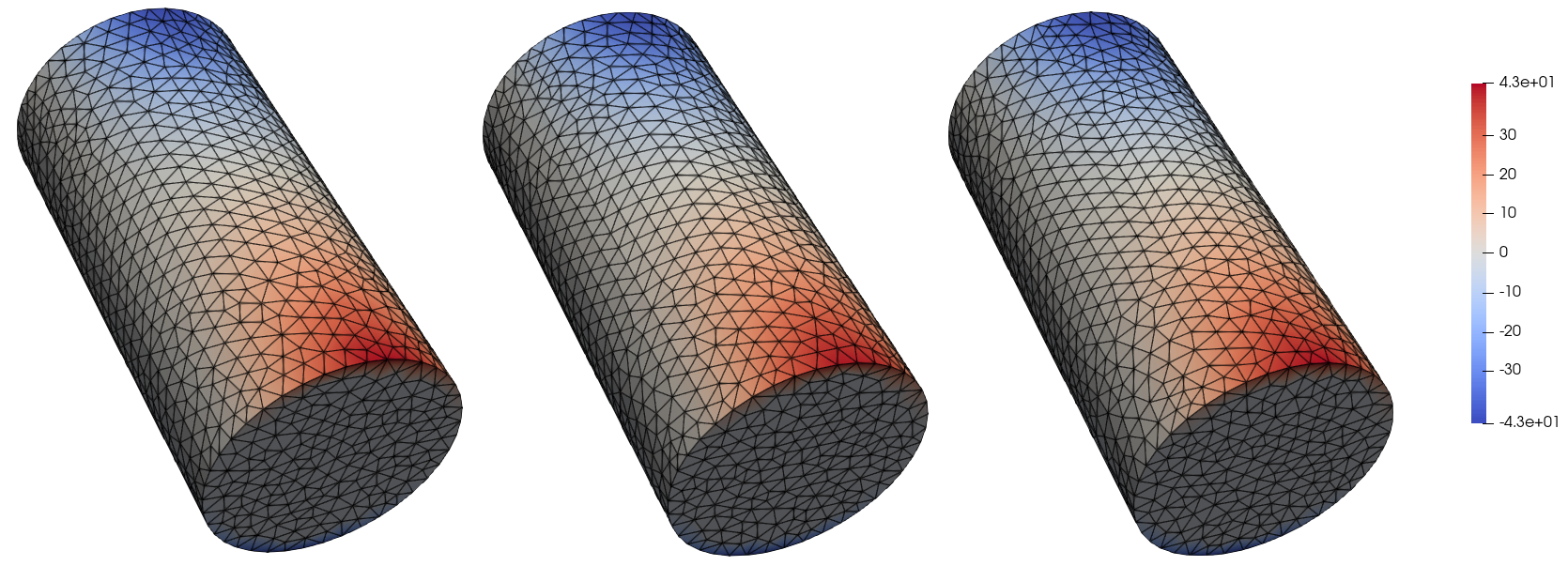}
    \caption{The normal contact loads for three different angles.}
    \label{fig:normals3}
\end{figure}

\begin{figure}
    \centering
    \includegraphics[width=\textwidth]{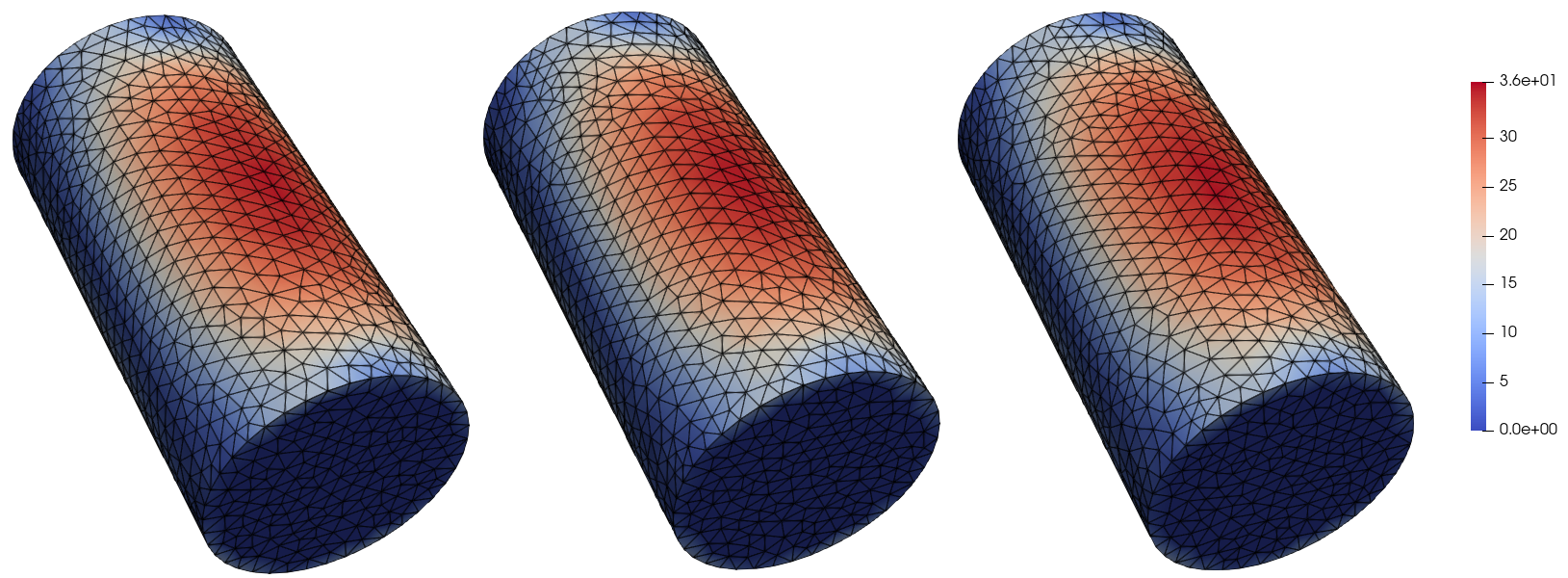}
    \caption{The tangential contact load magnitudes for three different angles.}
    \label{fig:tangents3}
\end{figure}

\begin{figure}
    \centering
    \includegraphics[width=\textwidth]{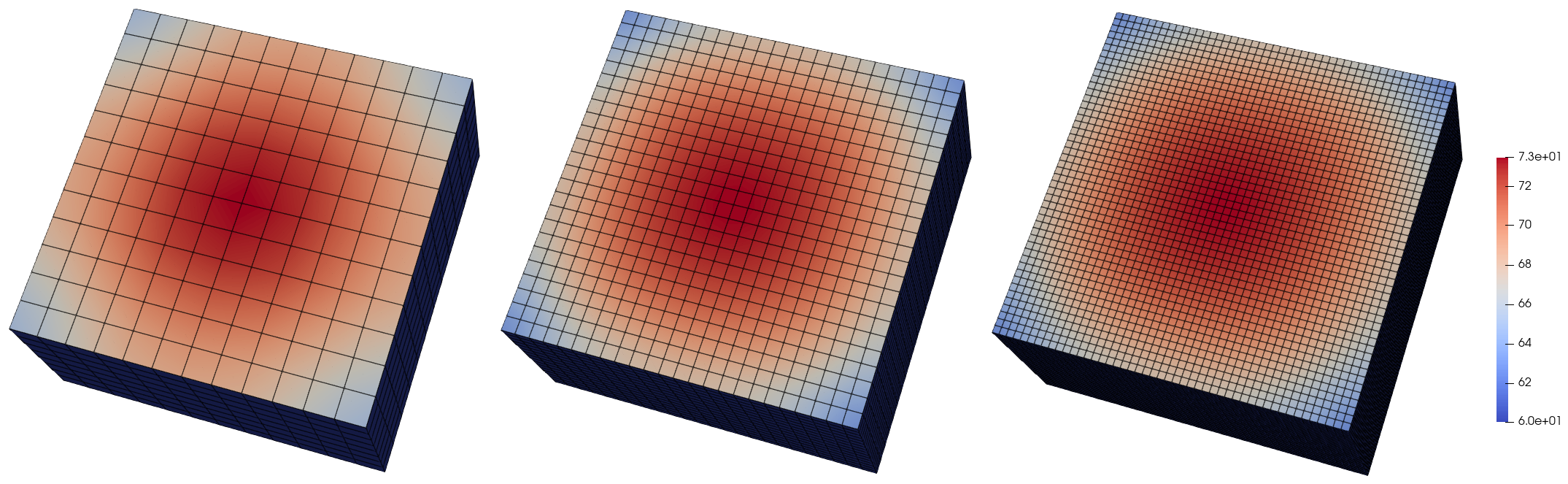}
    \caption{The normal contact loads for three different refinement levels.}
    \label{fig:normloads}
\end{figure}

\begin{figure}
    \centering
    \includegraphics[width=\textwidth]{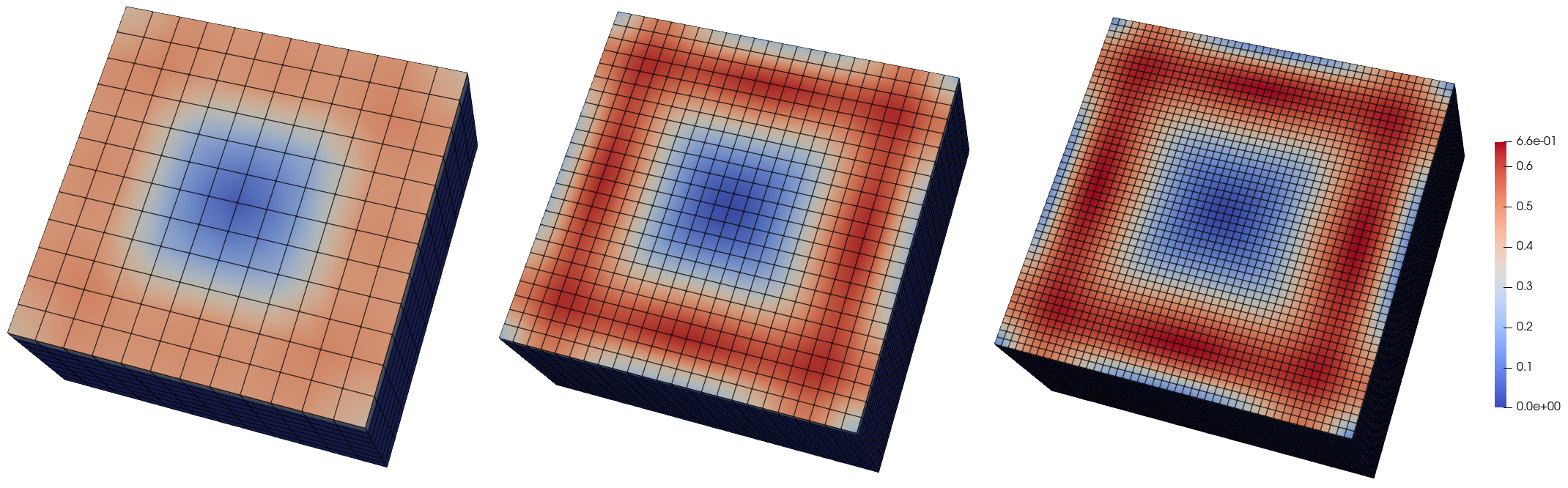}
    \caption{The tangential contact load magnitudes for three different refinement levels.}
    \label{fig:sliploads}
\end{figure}

\begin{figure}
    \centering
    \includegraphics[width=0.49\textwidth]{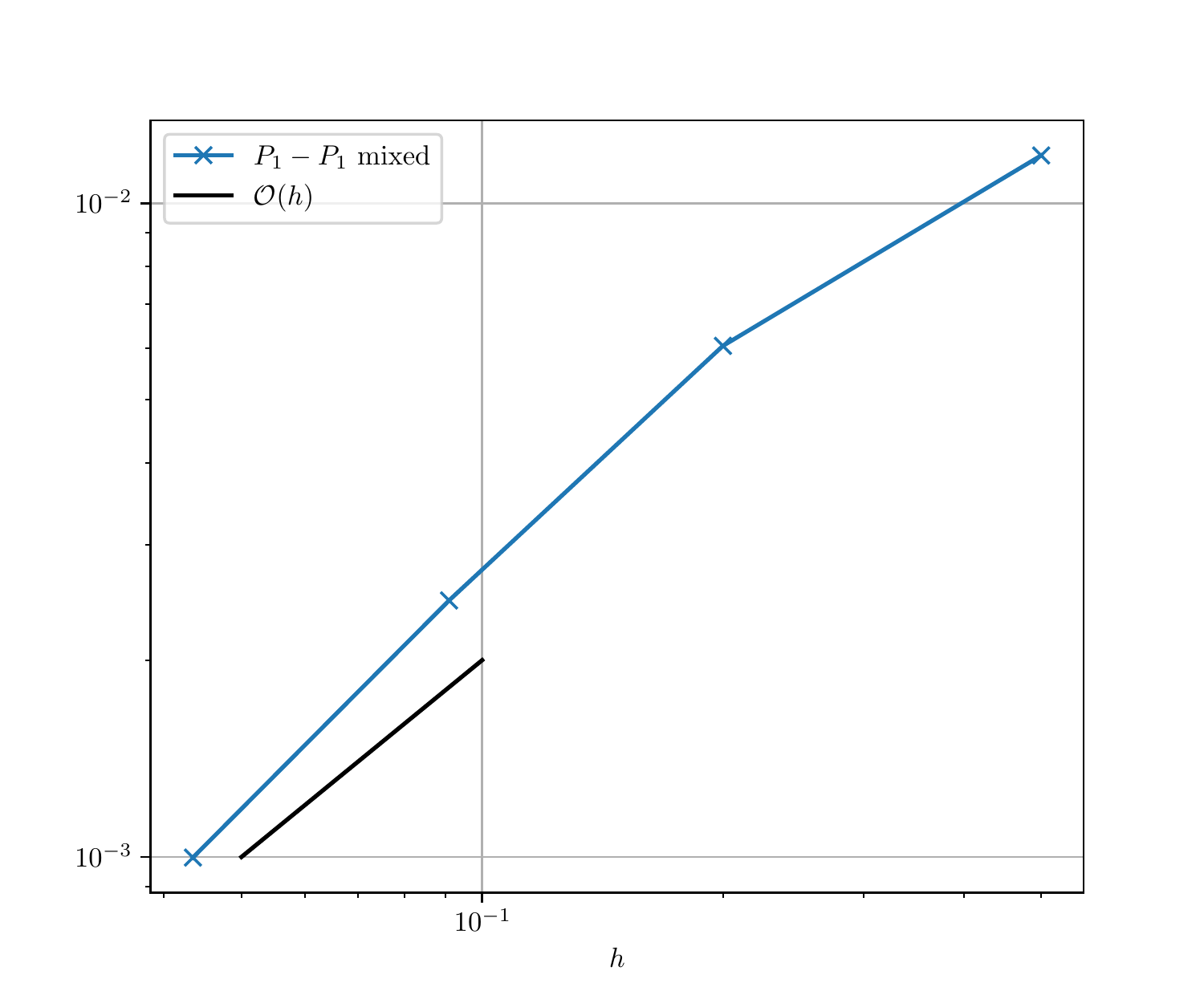}
    \caption{The absolute change in the strain energy between two subsequent refinement levels as a function of the mesh parameter $h$.}
    \label{fig:strainconv}
\end{figure}

\bibliography{references}
\bibliographystyle{elsarticle-num}

\end{document}